\documentclass[12pt]{article}
\usepackage{amssymb,amsmath,amsthm}
\usepackage{graphicx}
\usepackage{subcaption}
\usepackage{fullpage}
\usepackage{scrextend}
\usepackage{siunitx}
\usepackage{enumerate}
\usepackage{tikz,calc}
\usepackage{tikz-cd}
\usepackage{subcaption}
\usetikzlibrary{calc}
\usepackage{epsfig,graphicx}
\usepackage{listings}
\usepackage{enumerate}

 \newtheorem{theorem}{Theorem}

\newtheorem{lemma}[theorem]{Lemma}
\newtheorem{cor}[theorem]{Corollary}

\title{On the directions determined by a Cartesian product in an affine Galois plane}
\author{Daniel Di Benedetto\thanks{dibenedetto@math.ubc.ca} \qquad J\'ozsef Solymosi\thanks{solymosi@math.ubc.ca} \qquad Ethan P. White\thanks{epwhite@math.ubc.ca}\\
Department of Mathematics\\
The University of British Columbia \\
Vancouver, BC\\
 Canada V6T 1Z2}

\begin{document}

\maketitle

\abstract{We prove that the number of directions contained in a set of the form $A \times B \subset AG(2,p)$, where $p$ is prime, is at least $|A||B| - \min\{|A|,|B|\} + 2$. Here $A$ and $B$ are subsets of $GF(p)$ each with at least two elements and $|A||B| <p$. This bound is tight for an infinite class of examples. Our main tool is the use of the R\'edei polynomial with Sz\H{o}nyi's extension. As an application of our main result, we obtain an upper bound on the clique number of a Paley graph, matching the current best bound obtained recently by Hanson and Petridis.} 

\section{Introduction} 

Let $U$ be a subset of the Desargusian affine plane $AG(2,p)$, where $p$ is a prime number. A direction is \emph{determined} by $U$ if two points of $U$ lie on a line in that direction. We can coordinatize $AG(2,p)$ so that $U = \{(a_i,b_i) \colon 1 \leq i \leq |U|\}$, where $a_i,b_i \in GF(p)$ for all $1 \leq i \leq |U|$, and then the set of directions determined by $U$ is given by
\[ D = \left\{ \frac{b_i-b_j}{a_i-a_j} \colon 1 \leq i<j \leq n \right\}.\]
Note that $D$ is a subset of $GF(p)\cup \{\infty\}$. 

The possible values of $|D|$ have been studied by various authors. For a survey of results on this topic see \cite{ts2} and \cite{BB}. A key tool in this area are the properties of \emph{lacunary polynomials}, which are polynomials with several consecutive coefficients equal to zero. R\'edei's monograph \cite{r}, as well as Ball and Blokhuis's chapter \cite{BB} contain many results on lacunary polynomials and their applications, one of which is a sharp lower bound of $(p+3)/2$ on the size of $D$ for sets of size exactly $p$, excepting lines. The size of $|D|$ has also been considered in the setting $AG(2,q)$, $q$ a prime power, see for example \cite{bbbss}. 

R\'edei's method was later extended by Sz\H{o}nyi \cite{ts} to sets of size smaller than $p$. Our result uses Sz\H{o}nyi's extension of R\'edei's method and relies on the fact that in the case when the set is a Cartesian product, the relevant polynomials have a very special structure that can be exploited. We prove the following theorem, improving Sz\H{o}nyi's bound by a factor of two (up to lower order terms) for Cartesian product sets.

\begin{theorem}\label{t1}
Let $A,B\subset GF(p)$ be sets each of size at least two such that $|A||B| < p$. Then the set of points $A\times B\subset AG(2,p)$ determines at least $|A||B| - \min\{|A|,|B|\} + 2$ directions.
\end{theorem}

Let $d>1$ be a divisor of $p-1$ and let $Z_d$ be a multiplicative subgroup of size $d$ inside $GF(p)$. If a set $A$ satisfies $A-A\subset Z_d\cup\{0\}$, then all of its directions are elements of $Z_d\cup\{0,\infty\}$. Thus, as a corollary of Theorem 1, we obtain the following result, which was recently proved by Hanson and Petridis \cite{hp} using Stepanov's method.

\begin{cor}
Let $A\subset GF(p)$ be a set such that $A-A\subset Z_d\cup\{0\}$. Then
\[ |A|(|A|-1)\leq d.\]
\end{cor}

In particular, if $p$ is congruent to 1 modulo 4 and $d=(p-1)/2$ this gives an upper bound of $(\sqrt{2p-1}+1)/2$ on the clique number of the Paley graph $G_p$. Recall that the vertices of $G_p$ are the elements of $GF(p)$ with an edge between elements whose difference is a square in $GF(p)$.

Estimating the size of sets of the form $(A-A)/(A-A)$ played a crucial role in early sum-product estimates over finite fields, \cite{g,ks}, and it is still an important tool in proving sum-product type bounds (see e.g. in \cite{mprs}). For $A \subset GF(p)$ with $|A|^2 < p$, Theorem~\ref{t1} gives that 
\[ \#\left\{ \frac{a-b}{c-d} \colon (a,b,c,d) \in A^4, a \neq b, c\neq d \right\} \geq |A|^2-|A|.  \]
Hence there is a nonzero $x \in (A-A)/(A-A)$ such that the number of representations $x = (a-b)/(c-d)$ with $(a,b,c,d) \in A^4$ is at most $|A|^2-|A|$. Consider the set $A-x A = \{\alpha - x \beta \colon (\alpha,\beta) \in A^2 \}$. The number of representations of a $y \in \alpha - x \beta$ cannot be too high on average, otherwise $x$ would have many representations in $(A-A)/(A-A)$. This idea can be made rigorous using the Cauchy-Schwarz inequality. The corresponding result is recorded in the following corollary.

\begin{cor} Let $A$ be a subset of $GF(p)$ such that $|A|^2 <p$. There exist $a,b,c,d \in A$ such that $|(a-b)A + (c-d)A| \geq |A|^3/(2|A|-1)$. 
\end{cor}

Our arguments for Cartesian products can be extended to a set in $GF(p)$ consisting of a union of two Cartesian products. A corollary of this is as follows. 

\begin{cor}\label{cu} Let $A,B \subset GF(p)$ be disjoint sets each of size at least two such that each of the difference sets $A-A$, $A-B$, $B-B$ contain either only squares, or only non-squares, in addition to $0$. Then
\[ \min\{|A|^2-2|A|,|B|^2-2|B|\} + |A||B| + 2 \leq \frac{p+3}{2}.\]

\end{cor}

For a subset $A \subseteq GF(p)$, the directions determined by $A \times \{0,1\}$ is the set $(A-A)\cup \{\infty\}$. Hence by Theorem~\ref{t1} we recover an instance of the well known Cauchy-Davenport Theorem~\cite{D}.

\begin{cor} Let $A \subseteq GF(p)$ be nonempty, then $|A-A| \geq \min\{p,2|A|-1\}$. 
\end{cor}

\section{R\'edei polynomials}

Let $U= \{(a_i,b_i) \colon 1 \leq i \leq |U|\}$ be a subset of the affine plane $AG(2,p)$, and $D$ be the set of directions determined by $U$. Suppose that $AG(2,p)$ is coordinatized so that $\infty \in D$. Put $n = |U|$. The R\'edei polynomial of $U$ is
\[ H(x,y) = \prod_{i=1}^n (x+a_iy-b_i) .\]

Consider $H_y(x) = H(x,y)$ as a polynomial with indeterminate $x$ and coefficients in $GF(p)[y]$. Define the set $A_y = \{-a_iy+b_i\}_{i=1}^n$. Observe that $H_y(x)$ divides $x^p-x$ if and only if the elements of $A_y$ are all distinct, and this is equivalent to $y \not\in D$. In the case $y \not\in D$, we see that $(x^p-x)/H_y(x)$ has a root at every element of $GF(p) \setminus A_y$, i.e. the coefficients of $(x^p-x)/H_y(x)$ are symmetric polynomials of the form $\sigma_k(GF(p)\setminus A_y)$, $k = 1,2,\ldots,p-n$. We can determine the symmetric polynomials $\sigma_k(GF(p)\setminus A_y)$ in terms of the symmetric polynomials $\{\sigma_i( A_y)\}_{i=1}^k$ recursively as follows. 

For $1 \leq k< p-1$ we have $\sigma_k(GF(p)) = 0$ and so 
\[ \sum_{i=0}^k \sigma_i(A_y)\sigma_{k-i}(GF(p)\setminus A_y) = 0.\]
This gives, for example
\[ \sigma_1(GF(p)\setminus A_y) = -\sigma_1(A_y), \quad \text{and} \quad \sigma_2(GF(p)\setminus A_y) = \sigma_1^2(A_y)-\sigma_2(A_y).\]
Continuing recursively we see that $\sigma_k(GF(p)\setminus A_y)$ is a polynomial in $GF(p)[y]$ of degree at most $k$ and can be defined even when the elements of $A_y$ are not all distinct. Put $m=p-n$ and define
\begin{equation}\label{extpoly} f(x,y) = x^m - \sigma_1(GF(p) \setminus A_y)x^{m-1} + \sigma_2(GF(p) \setminus A_y)x^{m-2} + \cdots +(-1)^m \sigma_m(GF(p) \setminus A_y). \end{equation}
Note that $f$ is a degree $m$ polynomial in $GF(p)[x,y]$ and crucially we have 
\[ H(x,y)f(x,y) = x^p-x\]
for all $y \not\in D$. 

For more on the construction and properties of $H$ and $f$ see \cite{ts,ts2}. Let
\[ H(x,y)f(x,y) = x^p + h_1(y)x^{p-1} + h_2(y)x^{p-2} + \cdots + h_p(y),\]
and note that deg$(h_i) \leq i$. Since $H(x,y)f(x,y) = x^p-x$, for every $y \not\in D$ we see that if $i \neq p-1$ then $h_i(y) = 0$ for all $y \not\in D$. Recall that there are $p+1$ directions in $AG(2,p$). Since $\infty \in D$, there are $p+1-|D|$ directions not in $D$, and all such directions are in $GF(p)$. This implies that $h_i \equiv 0$ if $i< p+1-|D|$. Equivalently, if $h_i \not\equiv 0$, then $|D| \geq p+1-i$. Therefore showing that there is a high degree term in this polynomial with a nonzero coefficient results in a lower bound on $|D|$. This is how we will proceed.

\medskip
\medskip

\section{Directions in Cartesian products}

Let $U$ be a Cartesian product set in $AG(2,p)$, i.e. there exists a coordinatization such that $U = A \times B$, where $A,B \subset GF(p)$. Assume that the elements of $A$ and $B$ are all distinct, and put $|A|=m$, $|B|=n$. Let $A = \{ a_i \colon 1 \leq i \leq m\}$ and $B = \{b_j \colon 1 \leq  j \leq n\}$. 

If $m=1$ or $n=1$ then $U$ is contained in a line and spans only one direction. Notice also that any subset of $AG(2,p)$ with at least $p+1$ elements determines all directions. This is because there are only $p$ parallel lines in each direction, and so for each direction there must be a line in that direction containing at least two points from the set. Consequently, we will assume that $m,n \geq 2$ and $mn<p$. Translating preserves the number of directions, and so we will assume $0 \in B$. 

The R\'edei polynomial of $A \times B$ is 
\[ H(x,y) = \prod_{i,j } (x+a_iy-b_j) .\]
Let $A_y = \{-a_iy+b_j \colon 1\le i\le m; 1\le j\le n \}$. Put $k = p-mn$ and define
\[ f(x,y) = x^k - \sigma_1(GF(p) \setminus A_y)x^{k-1} + \sigma_2(GF(p) \setminus A_y)x^{k-2} + \cdots + (-1)^k\sigma_k(GF(p) \setminus A_y).\]
We will consider the R\'edei polynomial in the horizontal direction, $y = 0$. Let 
\begin{equation}\label{red0} H(x,0)f(x,0) = f(x,0)\prod_j (x-b_j)^m =  x^p + c_1x^{p-1} + c_2x^{p-2} + \cdots + c_p, 
\end{equation}
for some coefficients $c_1,\ldots,c_p \in GF(p)$. We will exploit the product structure of the polynomial above to obtain our result. We begin with a lemma.

\begin{lemma}\label{lem} Let $R,S \in GF(p)[x]$ be polynomials each with constant term equal to $1$ and $\deg{R} \geq 1$. Suppose that $R$ and $R'$ are relatively prime and that $R$ does not divide $S$. Then $x^{\text{deg}(R)+\text{deg}(S)+1}$ does not divide $R^m(x)S(x)-1$ for any positive integer $m$ such that $p$ does not divide $m$. 
\end{lemma}

\begin{proof} Suppose for a contradiction that there exist $R$, $S$, and $m$ satisfying the conditions of the Lemma. Then there exists a polynomial $P(x) \in GF(p)[x]$ such that 
\begin{equation}\label{lemeq} R^m(x) S(x) = 1 + x^{\text{deg}(R)+\text{deg}(S)+1}P(x) . 
\end{equation}
Let $k = \text{deg}(R)$ and $n = \text{deg}(S)$. By differentiating (\ref{lemeq}) we obtain 
\[ R^{m-1}(x) ( mR'(x)S(x) + R(x) S'(x) ) = x^{k+n} ((k+n+1)P(x) + xP'(x)).\]
Since the constant term in $R^{m-1}(x)$ is 1, we see that $x^{k+n}$ divides $mR'(x)g(x) + R(x) S'(x)$. But the degree of $mR'(x)S(x) + R(x) S'(x)$ is at most $k+n-1$ and so $mR'(x)S(x) + R(x) S'(x) = 0$. Since $R$ and $R'$ are relatively prime, it must be the case that $R$ divides $mS$. Since $m \neq 0$ in $GF(p)$ we see $R$ divides $S$, a contradiction. 

\end{proof}

\begin{proof}[Proof of Theorem~\ref{t1}]
Recall that if $c_i \neq 0$, then there are at least $p - i+1$ directions in $A \times B$. Suppose for a contradiction that $c_1 = c_2 = \ldots = c_{k+n-1} = 0$. Put $R(y) = \prod_{j=1}^n (1-b_jy)$, and $S(y) = y^kf(y^{-1},0)$. We see that $R(y),S(y) \in GF(p)[y]$, deg$(R) = n-1$ and deg$(S) \leq k$. Substitute $x = y^{-1}$ in (\ref{red0}) and multiply by $y^p$ to obtain
\[ R^m(y)S(y) = 1 + c_1y + c_2y^2 + \cdots + c_py^p = 1 + y^{k+n} Q(y),\]
for some polynomial $Q(y) \in GF(p)[y]$. Since the elements of $B$ are distinct, all roots of $R$ have multiplicity 1, and so $R$ is relatively prime to $R'$. Let $q$ be the highest power of $R$ dividing $S$. From the above we have 
\begin{equation*}\label{aplem} R^{m+q} \left( \frac{S}{R^q} \right) = 1+y^{k-q(n-1)+n}[y^{q(n-1)}Q(y)].\end{equation*}
We have the following relations on the above variables. 
\[ mn+k = p, \quad \text{and} \quad k-q(n-1) \geq 0.\]
It is easy to obtain the relation $m+q \leq p - m/(n-1)<p$ from the above. Therefore by Lemma~\ref{lem} we conclude that $\deg{R} = 0$, i.e. $R(y) = 1$. This gives $B = \{0\}$, which is a contradiction since we assumed $|B| \geq 2$. It follows that at least one of $c_1,\ldots,c_{k+n-1}$ is nonzero, and so there are at least $p-k-n+2 = mn-n+2$ directions in $A \times B$. By rotating the affine plane $90^{\circ}$ and repeating the argument we obtain the result
\begin{equation*}
\# \{ \text{Directions in } A \times B \} \geq  |A||B| - \min\{|A|,|B|\} + 2 .
\end{equation*}
\end{proof}

We remark that in the proof of Theorem~\ref{t1}, Lemma 6 could be substituted by a similar and simple different argument. We conclude with a proof of Corollary~\ref{cu}. 

\begin{proof}[Proof of Corollary~\ref{cu}] Let $A,B \subset GF(p)$ be as in the statement of the corollary. Let $|A| = m$, $|B|=n$, and put $A = \{a_i\}_{i=1}^m$, $B = \{b_j\}_{j=1}^n$, and $U =A \times A \cup B \times B$. Our strategy will be to bound the number of directions in $U$. Consequently, we can assume $0 \in A$ by translating $U$. Without loss of generality, we'll assume $m \geq n$. The R\'edei polynomial $H(x,y)$ of $U$ evaluated at $y=0$ is 
\[ H(x,0) = \prod_{i=1}^m (x-a_i)^m \prod_{j=1}^n (x-b_j)^n .\]
Let $f(x,y)$ be the polynomial defined in (\ref{extpoly}) corresponding to $U$. Define $k = \deg(f) = p-m^2-n^2$. Put $H(x,0)f(x,0) = x^p + G(x)$ for some $G[x] \in GF(p)[x]$. Define $R(y) = \prod_{i=1}^m (1-a_iy) \prod_{j=1}^n (1-b_jy)$, and $S(y) = y^kf(y^{-1},0) \left( \prod_{i=1}^m (1-a_iy) \right)^{m-n}$. Note that $R(y),S(y) \in GF(p)[y]$, $\deg(R) =m+n-1$, and $\deg{S}  \leq k + (m-1)(m-n)$. We make a similar substitution to that in the proof above as follows.

\[ y^p H(y^{-1},0)f(y^{-1},0) = R^n(y)S(y)  = 1 + y^p G(y^{-1}) .\]

Let $q$ be the highest power of $R$ dividing $S$. The above gives
\begin{equation}\label{coreq} R^{m+q} \left( \frac{S}{R^q} \right) =1 + y^p G(y^{-1}) .\end{equation}
The following relations hold. 
\[ m^2+n^2+k = p, \quad \text{and} \quad k+(m-1)(m-n) - q(m+n-1) \geq 0.\]
It is easy to check that the above implies $m+q <p$. The lowest degree term in $y^{p}G(y^{-1})$ is $p - \deg(G)$. Applying Lemma~\ref{lem} to (\ref{coreq}) gives  
\[ p - \deg{G} \leq k + (m-1)(m-n) + m+n-1.\]
Recall that the number of directions in $U$ is at least $\deg(G) +1$. Therefore $U$ determines at least $n^2+mn-2n+2$ directions. Every direction in $U$ is a quotient of two squares, or two non-squares in $GF(p)$. Hence all directions are squares, or zero, or infinity. This amounts to no more than $\frac{p+3}{2}$ directions, thereby giving the required result.

\end{proof}

\section{Concluding remarks}

The bound of Theorem~\ref{t1} is sometimes tight. For example if $p = 41$ and $A = \{0,1,5,9,10\}$, i.e. a maximal Paley clique, then the directions determined by $A \times A$ are the quadratic residues and $0$ and $\infty$. This totals $22$ directions, matching the lower bound $5^2 - 5 +2=22$ given by Theorem~\ref{t1}. Interestingly, this is the largest square grid we have found in which Theorem~\ref{t1} is tight. An infinite class of examples achieving exactly the lower bound are the long rectangles $A=\{0,1\}$, $B = \{0,1,\ldots,n-1\}$, and $p > 2n$ or $A = \{0,1,2\}$, $B = \{0,1,\ldots,n-1\}$, $n$ odd, and $p>3n$. 

It is worth noting that in the proof of Theorem~\ref{t1}, to show that the number of directions determined by $A \times B$ is at least $mn-n+2$, we used that the R\'edei polynomial at $y=0$ was of the form $H(x,0)=\prod_j (x-b_j)^m$. A set of points in $AG(2,p)$ has a R\'edei polynomial of this form if $n$ horizontal lines each contain exactly $m$ points, and so does not necessarily need to be a Cartesian product. For example, let $A = \{a_i\}_{i=1}^n$, $B = \{b_i\}_{i=1}^n$, be subsets of $GF(p)$ such that $n^2<p$ and $0 \not\in B$. Consider the following sets in $AG(2,p)$ 
\begin{enumerate}[i.]
    \item $\{(a_i+a_j^2 , a_j) \colon 1 \leq i,j \leq n\}$,
    \item $\{(b_i+b_j^{-1} , b_j) \colon 1 \leq i,j \leq n\}$.
\end{enumerate}
Note that the vertical direction is not necessarily determined by either of the above sets, and so the number of directions determined is only at least $n^2-n+1$. The below sets describe the reciprocal directions of the sets above, but exclude $\infty$ (formerly the direction $0$). Therefore the sets below each have size at least $n^2-n$. 
\begin{enumerate}[I.]
    \item $\{(x-y)(z-w)^{-1} + (z+w) \colon w,x,y,z \in A, z \neq w \}$,
    \item $\{(x-y)(z-w)^{-1} - (zw)^{-1} \colon w,x,y,z \in B, z \neq w \}$.
\end{enumerate}

\section{Acknowledgements} The research of the first author was supported in part by a Four Year Doctoral Fellowship from the University of British Columbia. The research of the second author was supported in part by an NSERC Discovery grant and OTKA K 119528 grant. The work of the second author was also supported by the European Research Council (ERC) under the European Union's Horizon 2020 research and innovation programme (grant agreement No. 741420, 617747, 648017). The research of the third author was supported in part by Killam and NSERC doctoral scholarships. We also thank Sammy Luo for helpful comments. 


\end{document}